\pdfoutput=1\relax
\documentclass[11pt, english]{article}
\usepackage{meta/preamble}
\usepackage{mathtools}

\title{Condensed Brown Comenetz Duality}
\date{\today}
\author{Roey Hel-Or, Amos Kaminski\thanks{Department of Mathematics, Weizmann Institute of Science \\ roey.hel-oer@weizmann.ac.il \\
amos.kaminski@weizmann.ac.il}}
\begin{document}

\setlength{\parindent}{0pt}

\maketitle
\begin{center}
\end{center}
\begin{abstract}
We construct an involutive duality on condensed spectra with locally compact homotopy groups, extending classical Pontryagin duality. We prove that dualization induces Pontryagin duality on homotopy groups, with degrees reversed, and that the canonical biduality map is an equivalence. In particular, every ordinary spectrum, viewed as a discrete condensed spectrum, satisfies biduality, without any finiteness assumptions on its homotopy groups.
\end{abstract}

\setcounter{tocdepth}{2}
$\\$
\begin{figure}[H]
  \centering{}
  \setlength{\fboxsep}{-5pt}
  \setlength{\fboxsep}{5pt}
  \frame{\includegraphics[scale=0.14]{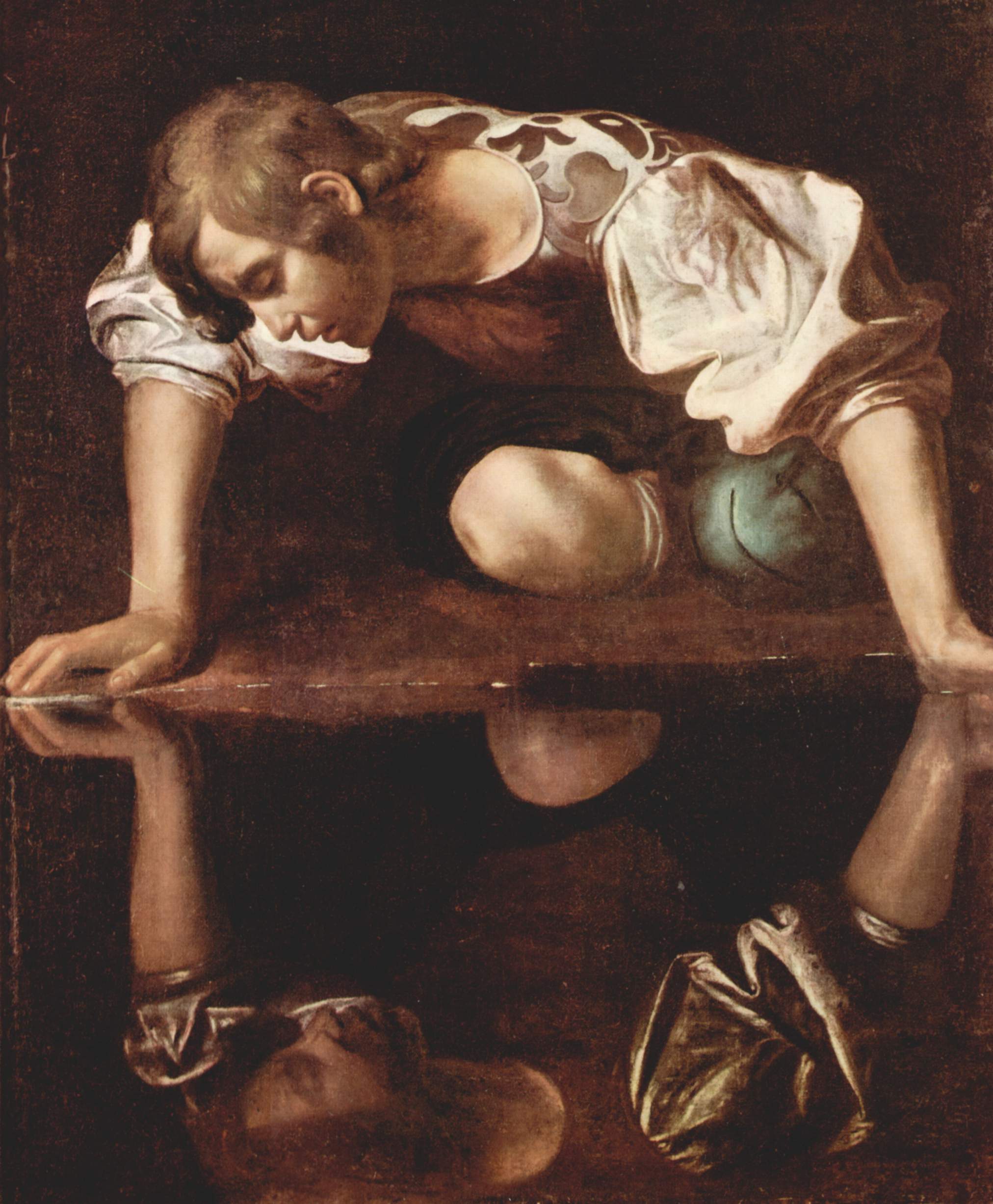}}
  \caption{\footnotesize 
le Narcisse du Caravage\\}
    
\end{figure}
\newpage
\begingroup
\renewcommand{\baselinestretch}{1.2}\selectfont
\setlength{\parskip}{6pt}
\tableofcontents
\endgroup
\newpage
\section{Introduction}

This result addresses a difficulty already identified by Brown and Comenetz
at the outset of their theory. In their 1974 account, they explicitly set
aside the natural topology on the character groups and explained that their
inability to incorporate it led to finiteness assumptions in most of their
results \cite[Introduction]{BC74}. The issue was already present with
$\mathbb R/\mathbb Z$ as the character group: its topology was forgotten in
passing to spectra. Condensed spectra allow us to retain this topology and
establish biduality for locally compact homotopy groups. 

\subsection{Motivation and statement of the main result}

Pontryagin duality and Brown--Comenetz duality are two instances of the same
principle: an object can sometimes be recovered from its characters into a
suitable dualizing object.  The purpose of this paper is to construct a common
generalization of these two dualities in the category of condensed spectra.

Let \(A\) be a locally compact abelian group.  Its Pontryagin dual is
\[
\widehat{A}
=
\Hom_{\mathrm{cont}}
\left(A,\mathbb{T}\right),
\]
endowed with the compact-open topology, where $\mathbb{T}=\mathbb{R}/\mathbb{Z}$ is the circle group. The Pontryagin duality
theorem states that \(\widehat{A}\) is again locally compact and that the evaluation
map
\[
d_A:A\to
\widehat{\widehat{A}}
\]
is an isomorphism.  Moreover, Pontryagin duality exchanges compact and
discrete abelian groups.  We use the formulation in
\cite[Theorem~4.1(ii)--(iii)]{ScholzeCondensed2026}.  If topological abelian groups
are regarded as condensed abelian groups, then the condensed internal Hom
recovers the compact-open topology: for a compactly generated Hausdorff
topological abelian group \(A\) and a Hausdorff topological abelian group \(B\),
there is a natural isomorphism
\[
\uHom_{\Cond(\Ab)}(A,B)
\cong
\Hom_{\mathrm{cont}}(A,B)
\]
of condensed abelian groups, where the right-hand side is the condensed abelian group associated with the
compact-open topology
\cite[Proposition~4.2]{ScholzeCondensed2026}. Thus, Pontryagin duality can be written
intrinsically in condensed mathematics as
\[
\widehat{A}
=
\uHom_{\Cond(\Ab)}
\left(A,\mathbb{T}\right).
\]

For finite abelian groups, Pontryagin duality reduces to the usual character duality
\[
A\mapsto\Hom(A,\mathbb{Q}/\mathbb{Z}),
\]
since every character
has finite image and \(\mathbb{T}_{\mathrm{tors}}\cong\mathbb{Q}/\mathbb{Z}\).

\medskip

There is a parallel construction in stable homotopy theory.  Since
\(\mathbb{Q}/\mathbb{Z}\) is injective in \(\Ab\), it admits an
injective lift, relative to the standard \(t\)-structure on spectra, to a spectrum
\(I_{\mathbb{Q}/\mathbb{Z}}\).  In Lurie's terminology, this follows from the
classification of injective objects in a stable category with a \(t\)-structure
\cite[Definition~C.5.7.2, Proposition~C.5.7.3, and
Theorem~C.5.7.4]{SAG}.

We write
\[
    \widehat{X}
    :=
    \uHom_{\Sp}
    \left(X,I_{\mathbb{Q}/\mathbb{Z}}\right)
\]
for the Brown--Comenetz dual of a spectrum \(X\).  The resulting spectrum
satisfies
\[
\pi_n\widehat{X}
\cong
\uHom_{\Ab}
\left(\pi_{-n}X,\mathbb{Q}/\mathbb{Z}\right)
\]
naturally in every spectrum \(X\) and every integer \(n\).  If all homotopy
groups of \(X\) are finite, the evaluation map
\[
X\longrightarrow\widehat{\widehat{X}}
\]
is an equivalence; this is the Brown--Comenetz duality theorem
\cite[Theorem~7.4.1]{RavenelNilpotence}.

Passing to condensed spectra allows us to remove this finiteness hypothesis. For the duality constructed below, every ordinary spectrum, regarded as a discrete condensed spectrum, is canonically equivalent to its double dual. Its homotopy groups may be arbitrary abelian groups: their discrete topology makes them locally compact, while their duals carry the compact topology supplied by Pontryagin duality.
We seek a dualizing spectrum in
\[
\Cond(\Sp)
\]
whose effect on homotopy groups is Pontryagin duality.  Let
\[
\Cond(\Sp)_{\mathrm{lc}}
\subseteq
\Cond(\Sp)
\]
denote the full subcategory of condensed spectra \(X\) such that
\(\pi_nX\) is a locally compact abelian group for every \(n\in\mathbb{Z}\).
The desired object \(I_{\mathbb{T}}\) should satisfy
\[
\pi_n
\uHom_{\Cond(\Sp)}
\left(X,I_{\mathbb{T}}\right)
\cong
\uHom_{\Cond(\Ab)}
\left(\pi_{-n}X,\mathbb{T}\right)
\]
for every \(X\in\Cond(\Sp)_{\mathrm{lc}}\).
This would simultaneously refine the Brown--Comenetz formula in the topological
direction and the Pontryagin duality formula in the stable direction.

There is an immediate obstruction to constructing
\(I_{\mathbb{T}}\) as an injective lift: the condensed circle
\(\mathbb{T}\) is not injective in \(\Cond(\Ab)\).  Nevertheless,
this failure of injectivity disappears when testing against locally compact sources.  More
precisely, Hoffmann and Spitzweck prove that, for every locally compact abelian
group \(A\),
\[
\Ext^{q}_{\operatorname{LCA}}
\left(A,\mathbb{T}\right)=0
\qquad (q\geq 1)
\]
\cite[Proposition~4.14(vii)]{HoffmannSpitzweck2007}.  The comparison functor
\[
D^{b}\left(\operatorname{LCA}\right)
\longrightarrow
D\left(\Cond(\Ab)\right)
\]
is fully faithful
\cite[Corollary~4.9]{ScholzeCondensed2026}, and the proof identifies the
corresponding derived internal Hom objects
\cite[Corollary~4.9, proof]{ScholzeCondensed2026}.  It is in this sense that, although
\(\mathbb{T}\) is not injective in \(\Cond(\Ab)\), it behaves as an
injective object when tested against locally compact sources.

We therefore construct the desired spectrum by gluing the ordinary
Brown--Comenetz spectrum to the Eilenberg--Mac Lane spectrum of the circle.
\medskip

Let
\[
\iota:
\left(\mathbb{Q}/\mathbb{Z}\right)^{\mathrm{disc}}
\longrightarrow
\mathbb{T}
\]
be the natural inclusion, and let
\[
h^{\mathrm{disc}}:
H\left(\left(\mathbb{Q}/\mathbb{Z}\right)^{\mathrm{disc}}\right)
\longrightarrow
I_{\mathbb{Q}/\mathbb{Z}}^{\mathrm{disc}}
\]
be the discrete condensation of the map corresponding to
\[
\id_{\mathbb{Q}/\mathbb{Z}}
\in
\End_{\Ab}
\left(\mathbb{Q}/\mathbb{Z}\right)
\]
under the Brown--Comenetz universal property.  We define
\[
    I_{\mathbb{T}}\in\Cond(\Sp)
\]
by the following pushout square:
\[\begin{tikzcd}
	{H\left(
	    \left(\mathbb{Q}/\mathbb{Z}\right)^{\mathrm{disc}}
	\right)} & {H\mathbb{T}} \\
	{I_{\mathbb{Q}/\mathbb{Z}}^{\mathrm{disc}}} & {I_{\mathbb{T}}.}
	\arrow["{H\iota}", from=1-1, to=1-2]
	\arrow["{h^{\mathrm{disc}}}"', from=1-1, to=2-1]
	\arrow[from=1-2, to=2-2]
	\arrow[from=2-1, to=2-2]
	\arrow["\lrcorner"{anchor=center, pos=0.125, rotate=180}, draw=none, from=2-2, to=1-1]
\end{tikzcd}\]

The two terms on the right and lower left have complementary virtues:
\(H\mathbb{T}\) has the desired degree-zero homotopy group,
whereas \(I_{\mathbb{Q}/\mathbb{Z}}^{\mathrm{disc}}\) has the desired stable
character-theoretic universal property.  The pushout joins them along their common
discrete torsion part.

For
\(X\in\Cond(\Sp)_{\mathrm{lc}}\), define its
Brown--Comenetz--Pontryagin dual by
\[
    \widehat{X}
    :=
    \underline{\Hom}_{\Cond(\Sp)}
    \left(X,I_{\mathbb{T}}\right).
\]

The main result is the following.

\begin{theorem*}[Condensed Brown--Comenetz--Pontryagin duality]
\label{thm:introduction-main}
With \(I_{\mathbb T}\) and \(\widehat{(-)}\) as above, the following hold:

\begin{enumerate}[label=\textup{(\roman*)}]
\item
The spectrum \(I_{\mathbb{T}}\) is coconnective, and there is a canonical
isomorphism
\[
\pi_0 I_{\mathbb{T}}\cong\mathbb{T}.
\]

\item
For every
\(X\in\Cond(\Sp)_{\mathrm{lc}}\) and \(n\in\mathbb{Z}\), there is a natural
isomorphism
\[
\pi_n\widehat{X}
\xrightarrow{\ \sim\ }
\widehat{\pi_{-n}X}.
\]

\item
The canonical evaluation map
\[
d_X:X\longrightarrow\widehat{\widehat{X}}
\]
is an equivalence for every
\(X\in\Cond(\Sp)_{\mathrm{lc}}\). Consequently,
\[
\widehat{(-)}:
\Cond(\Sp)_{\mathrm{lc}}^{\mathrm{op}}
\xrightarrow{\ \sim\ }
\Cond(\Sp)_{\mathrm{lc}}
\]
is a contravariant equivalence whose quasi-inverse is itself.
\end{enumerate}
\end{theorem*}
The theorem can be read in two ways.  It is a stable refinement of Pontryagin
duality because its effect on every homotopy group is the compact-open Pontryagin
dual.  It is also a locally compact refinement of Brown--Comenetz duality because
its homotopy formula has exactly the same form as the classical Brown--Comenetz
formula, with \(\mathbb{Q}/\mathbb{Z}\) replaced by the condensed circle
\(\mathbb{T}\).
\paragraph{Related work}
During the final preparation of this paper, we became aware of forthcoming work by Thomas Nikolaus and Phil Pützstück on condensed Anderson duality and reflexivity. The duality theorem for condensed spectra with locally compact homotopy groups proved here is also established in their work, with a different proof. The two projects were developed independently.
\subsection*{Acknowledgments}
We are deeply grateful to our advisor, Shachar Carmeli, for his guidance and support, and to Lior Yanovski for suggesting this research project. We thank Akhil Mathew and Dustin Clausen for insightful discussions on condensed mathematics. We also thank the members of our research group for their support and for the many conversations through which the ideas in this paper took shape.
Finally, we acknowledge the use of ChatGPT (OpenAI) and Gemini (Google) for assistance in exploring ideas and improving the clarity of the exposition. We have independently verified all results and proofs and take full responsibility for the content.

\subsection*{Funding}
This work was partially supported by BSF Grant No. 2024766 and by the Azrieli Foundation.

\section{Injective objects and injective lifts}

\begin{notation}

For an abelian category \(\mathcal A\), we use the notation
\[
    [X,Y]_{\mathcal A}:=\uHom_{\mathcal A}(X,Y).
\]
For a stable \(\infty\)-category \(\mathcal C\), we will use
\[
    [X,Y]_{\mathcal C}
    :=
    \Hom_{h\mathcal C}(X,Y)
    =
    \pi_0\Map_{\mathcal C}(X,Y).
\]
In either case, \([X,Y]\) denotes the abelian group of morphisms from \(X\) to \(Y\).

\end{notation}

We begin by recalling the ordinary notion of injectivity in an abelian
category and its analogue for a stable \(\infty\)-category equipped with
a \(t\)-structure.

\begin{definition}[Injective objects in an abelian category]
Let \(\A\) be an abelian category. An object \(J\in\A\) is called
\emph{injective} if the contravariant functor
\[
    [-,J]_\A\colon \A^{\mathrm{op}}\longrightarrow \Ab
\]
is exact.

We denote by
\[
    \operatorname{Inj}(\A)\subseteq \A
\]
the full subcategory spanned by the injective objects.
\end{definition}

In a stable \(\infty\)-category, the preceding definition cannot be
used without additional structure. Indeed, for every object \(Q\) in a
stable \(\infty\)-category \(\C\), the representable functor
\[
    \Map_{\C}(-,Q)\colon \C^{\mathrm{op}}\longrightarrow
    \mathrm{Spaces}
\]
carries cofiber sequences to fiber sequences. Consequently, exactness
of a representable functor does not distinguish a special class of
objects. The appropriate notion of injectivity therefore depends on a
chosen \(t\)-structure.
We will first recall the notion of the heart of a $t$-structure
\begin{definition}[The heart of a \(t\)-structure]
Let \(\C\) be a stable \(\infty\)-category equipped with a
\(t\)-structure
\[
    \bigl(\C_{\geq 0},\C_{\leq 0}\bigr)
\]
in the sense of
\cite[Definition~1.2.1.4]{HA}. Its \emph{heart} is the full
subcategory
\[
    \C^{\heartsuit}
    :=
    \C_{\geq 0}\cap\C_{\leq 0}.
\]

The truncation functors determine a functor
\[
    \pi_0\colon \C\longrightarrow\C^{\heartsuit}
\]
given by either of the canonically equivalent composites
\[
    \pi_0
    :=
    \tau_{\leq 0}\tau_{\geq 0}
    \simeq
    \tau_{\geq 0}\tau_{\leq 0}.
\]
More generally, using the homological indexing convention, we define
\[
    \pi_n(X)
    :=
    \pi_0\bigl(\Sigma^{-n}X\bigr)
    \qquad
    (n\in\mathbb{Z}).
\]

The category \(\C^{\heartsuit}\) is an ordinary abelian category; see
\cite[Definition~1.2.1.11 and Remark~1.2.1.12]{HA}.
\end{definition}
We will now define the notion of injective objects relative to a $t$-structure
\begin{definition}[\(t\)-injective objects]
An object \(Q\in\C\) is called \emph{\(t\)-injective} if
\[
    Q\in\C_{\leq 0}
\]
and
\[
    [X,\Sigma Q]_{\C}=0
\]
for every \(X\in\C_{\leq 0}\).

We denote by
\[
    \operatorname{Inj}_{t}(\C)\subseteq\C
\]
the full subcategory spanned by the \(t\)-injective objects.

Lurie calls these objects simply injective; the expression
\(t\)-injective emphasizes their dependence on the chosen
\(t\)-structure. See
\cite[Definition~C.5.7.2]{SAG}.
\end{definition}

The preceding vanishing condition has the following useful universal
reformulation.

\begin{proposition}[Lurie]
\label{prop:characterization-t-injective}
For \(X, Q\in\C\), denote by
\[
    \rho_{X,Q}: [X,Q]_\C
    \longrightarrow
    \bigl[\pi_0X,\pi_0Q\bigr]_{\C^{\heartsuit}},
\]
the canonical map of abelian groups induced by $\pi_0$. Then the following conditions are equivalent:
\begin{enumerate}[\textup{(1)}]
    \item
    The object \(Q\) is \(t\)-injective.

    \item
    $\rho_{X,Q}$ is an isomorphism for every \(X\in\C\)
\end{enumerate}
\end{proposition}

\begin{proof}
This is
\cite[Proposition~C.5.7.3]{SAG}.
\end{proof}

Thus a \(t\)-injective object \(Q\) represents the functor obtained by
first taking the degree-zero homotopy object and then mapping into the
injective object \(\pi_0Q\) of the heart.

\medskip
We will occasionally omit the subscripts from \(\rho_{X,Q}\)
when the objects \(X\) and \(Q\) are clear from context.

\begin{theorem}[Lurie]
\label{thm:injective-lifts}
Let \(\mathcal P\) be a Grothendieck prestable
\(\infty\)-category, and let
\[
    \C := \operatorname{Sp}(\mathcal P)
\]
be equipped with its canonical \(t\)-structure.
Then the functor \(\pi_0\) induces an equivalence of ordinary categories
\[
    \pi_0\colon
    h\operatorname{Inj}_{t}(\C)
    \xrightarrow{\ \sim\ }
    \operatorname{Inj}\bigl(\C^{\heartsuit}\bigr).
\]

In particular, every injective object \(J\in\C^{\heartsuit}\)
admits a \(t\)-injective lift \(I_J\in\C\), together with
an isomorphism
\[
    \pi_0 I_J \cong J.
\]
Such a lift, equipped with this identification, is unique
up to unique isomorphism in \(h\C\) compatible with the
identification.

Moreover, for every \(X\in\C\), the canonical map
\(\rho_{X,I_J}\), followed by the chosen identification
\(\pi_0 I_J\cong J\), gives an isomorphism
\[
    [X,I_J]_\C
    \xrightarrow{\ \sim\ }
    \bigl[\pi_0X,J\bigr]_{\C^{\heartsuit}}.
\]
This universal property characterizes \(I_J\) together
with its identification \(\pi_0I_J\cong J\).
\end{theorem}

\begin{proof}
The equivalence is
\cite[Theorem~C.5.7.4]{SAG}. The final isomorphism follows from
\Cref{prop:characterization-t-injective}.
\end{proof}

\subsection{The classical Brown-Comenetz spectrum}

The stable \(\infty\)-category \(\Sp\) carries its standard
\(t\)-structure
\[
    \Sp_{\geq 0}
    =
    \left\{
        X\in\Sp
        \,\middle|\,
        \pi_nX=0\text{ for every }n<0
    \right\}
\]
and
\[
    \Sp_{\leq 0}
    =
    \left\{
        X\in\Sp
        \,\middle|\,
        \pi_nX=0\text{ for every }n>0
    \right\}.
\]
Its heart is canonically equivalent to the category of abelian groups:
\[
    \Sp^{\heartsuit}\simeq\Ab.
\]

The abelian group \(\mathbb{Q}/\mathbb{Z} \in \Ab\) is injective. By
\Cref{thm:injective-lifts}, there exists an essentially unique
\(t\)-injective spectrum.

\begin{definition}[Brown--Comenetz spectrum]
The \emph{Brown--Comenetz spectrum} is the essentially unique
\(t\)-injective spectrum
\[
    I_{\mathbb{Q}/\mathbb{Z}}\in\Sp
\]
such that
\[
    \pi_0I_{\mathbb{Q}/\mathbb{Z}}
    \cong \mathbb{Q}/\mathbb{Z}.
\]
\end{definition}

By \Cref{thm:injective-lifts}, it is characterized by the natural
isomorphisms
\[\rho: 
    [X,I_{\mathbb{Q}/\mathbb{Z}}]_{\Sp}
    \xrightarrow{\ \sim\ }
    [\pi_0X,\mathbb{Q}/\mathbb{Z}]_{\Ab}
\]
for every \(X\in\Sp\). Applying this to \(\Sigma^nX\), for any
\(n\in\mathbb{Z}\), gives
\[
    [\Sigma^nX,I_{\mathbb{Q}/\mathbb{Z}}]_{\Sp}
    \xrightarrow{\ \sim\ }
    [\pi_{-n}X,\mathbb{Q}/\mathbb{Z}]_{\Ab}.
\]
See \cite[Example~C.5.7.10]{SAG}.

\begin{remark}
The spectrum \(I_{\mathbb{Q}/\mathbb{Z}}\) is not the Eilenberg--Mac
Lane spectrum \(H(\mathbb{Q}/\mathbb{Z})\). The former represents the
functor
\[
    X\longmapsto
    \bigl[\pi_0X,\mathbb{Q}/\mathbb{Z}\bigr]_{\Ab}
\]
on all spectra, whereas the latter represents ordinary 
cohomology with coefficients in \(\mathbb{Q}/\mathbb{Z}\).
\end{remark}

\section{Condensation}

\subsection{Condensed spectra}

Condensed mathematics provides a framework in which topological objects can be studied using algebraic and homological methods. Rather than thinking of a topological space only in terms of its underlying set and open subsets, one records its continuous families of points parametrized by profinite sets, and hence its notion of convergence. One can apply this construction to obtain the category of condensed spectra. Ordinary spectra embed into this category as discrete condensed spectra, while a general condensed spectrum has homotopy groups valued in condensed abelian groups. In this sense, condensed spectra generalize ordinary spectra and allow their homotopy groups to carry genuinely topological information, extending the role played by topological abelian groups in classical homotopy theory.
We leave the choices of cardinals and universes implicit, as our constructions apply in both the condensed and pyknotic settings.

We briefly recall the definitions needed in what follows.

\begin{definition}[The pro-\'etale site of a point]
Let
\[
    \mathrm{ProFin}
\]
denote the category of profinite sets and continuous maps. The
\emph{pro-\'etale site of a point}, denoted
\[
    \ast_{\mathrm{pro\acute{e}t}},
\]
is the category \(\mathrm{ProFin}\), equipped with the Grothendieck
topology in which a finite family
\[
    \{S_i\longrightarrow S\}_{i=1}^{n}
\]
is a covering family if the induced map
\[
    \coprod_{i=1}^{n}S_i\longrightarrow S
\]
is surjective; see \cite[Definition~1.2]{ScholzeCondensed2026}.
\end{definition}

\begin{definition}[Condensed objects]
For an ordinary category \(\mathcal{E}\) with the required limits, a
condensed \(\mathcal{E}\)-valued object is an \(\mathcal{E}\)-valued sheaf
on the pro-\'etale site of a point. We write
\[
  \Cond(\mathcal{E})
  :=\operatorname{Shv}(\operatorname{\ast_{\mathrm{pro\acute{e}t}}},\mathcal{E}).
\]
In particular, \(\Cond(\Ab)\) is the category of condensed abelian groups;
see \cite[Definition~1.2]{ScholzeCondensed2026}. 

\end{definition}

The category \(\Cond(\Ab)\) is Grothendieck abelian. In particular, it admits all small limits and colimits, has exact filtered colimits, and has enough injectives; see \cite[Theorem~1.10]{ScholzeCondensed2026}.

\begin{definition}[Condensed spectra]

For anima and spectra, we use the hypercomplete version of this
construction:
\[
  \Cond(\mathcal{S})
  :=\operatorname{Shv}^{\mathrm{hyp}}
    (\operatorname{\ast_{\mathrm{pro\acute{e}t}}},\mathcal{S}).
\]
This convention agrees with the pyknotic presentation by hypersheaves on
compact Hausdorff spaces; see
\cite[Proposition~2.3]{ScholzeCondensed2026} and
\cite[Corollary~2.4.4]{BarwickHaine2019}.
The hypercompleteness convention is essential below, as it ensures that
homotopy objects detect equivalences.

Equivalently, since stabilization of hypercomplete sheaves is computed pointwise, we can define the \(\infty\)-category of condensed spectra as the stabilization
\[
  \Cond(\Sp)
  :=\Sp\bigl(\Cond(\mathcal{S})\bigr).
\]
\end{definition}

\begin{definition}[Homotopy groups of a condensed spectrum]
Let $X\in 
    \Cond(\Sp)$.
    
For every integer \(n\), the \(n\)-th homotopy group of \(X\) is the
condensed abelian group
\[
    \pi_nX\in\Cond(\Ab)
\]
obtained by taking the sheaf associated with the presheaf
\[
    S\longmapsto\pi_n\bigl(X(S)\bigr).
\]
\end{definition}

The homotopy groups determine the standard \(t\)-structure on
\(\Cond(\Sp)\):
\[
    \Cond(\Sp)_{\geq 0}
    :=
    \left\{
        X\in\Cond(\Sp)
        \,\middle|\,
        \pi_nX=0
        \text{ for every }n<0
    \right\}
\]
and
\[
    \Cond(\Sp)_{\leq 0}
    :=
    \left\{
        X\in\Cond(\Sp)
        \,\middle|\,
        \pi_nX=0
        \text{ for every }n>0
    \right\}.
\]

Its heart is canonically equivalent to the category of condensed
abelian groups:
\[
    \Cond(\Sp)^{\heartsuit}
    \simeq
    \Cond(\Ab).
\]

\begin{definition}[The condensed object associated with a topological space]
Let \(T\) be a topological space. Its associated condensed set,
denoted
\[
    \underline{T},
\]
is defined on a profinite set \(S\) by
\[
    \underline{T}(S)
    :=
    \operatorname{Cont}(S,T),
\]
the set of continuous maps from \(S\) to \(T\).

If \(T\) is a topological group, ring, or abelian group, then
\(\underline{T}\) is naturally a condensed group, ring, or abelian
group, respectively.
\end{definition}

For compactly generated weak Hausdorff spaces, the functor
\[
    T\longmapsto\underline{T}
\]
is fully faithful; see
\cite[Theorem~2.16(ii)]{ScholzeCondensed2026}. In particular, locally
compact Hausdorff abelian groups embed fully faithfully into
\(\Cond(\Ab)\).

We will generally suppress the underline when viewing a locally compact abelian group as a condensed abelian group. Thus
\[
    \mathbb{T}:=\mathbb{R}/\mathbb{Z}
\]
denotes both the topological circle group and the corresponding condensed abelian group, with the intended meaning determined by context.

\begin{definition}[Discrete condensed objects]
Let \(M\) be a set, group, ring, or abelian group, equipped with the
discrete topology. Its associated condensed object is denoted
\[
    M^{\mathrm{disc}}.
\]

For every profinite set \(S\), one has
\[
    M^{\mathrm{disc}}(S)
    =
    \operatorname{Cont}(S,M),
\]
where the right-hand side consists of the locally constant maps
\(S\to M\).
\end{definition}

In particular,
\[
    \left(\mathbb{Q}/\mathbb{Z}\right)^{\mathrm{disc}}
    \in\Cond(\Ab)
\]
denotes the discrete condensed abelian group associated with the
abstract group \(\mathbb{Q}/\mathbb{Z}\). It is important to
distinguish this object from
\[
    \mathbb{T}=\mathbb{R}/\mathbb{Z}\in\Cond(\Ab),
\]
which carries its usual compact topology.

The inclusion of the discrete rational numbers into the real numbers induces a
monomorphism in condensed abelian groups
\[
    \iota:
    \left(\mathbb{Q}/\mathbb{Z}\right)^{\mathrm{disc}}
    \hookrightarrow
    \mathbb{T}.
\]

\begin{definition}[Discrete condensed spectra]
    There is a fully faithful functor
\[
    (-)^{\mathrm{disc}}:
    \Sp
    \longrightarrow
    \Cond(\Sp),
\]
left adjoint to the global sections functor. Explicitly, for a spectrum
\(E\) and a profinite set \(S\),
\[
    E^{\mathrm{disc}}(S) = \operatorname*{colim}_{S\twoheadrightarrow S_0}
    E^{S_0},
\]
where \(S_0\) ranges over the finite continuous quotients of \(S\). 

we call $E^{\mathrm{disc}}$ the \emph{discrete condensed spectrum} associated to $E$.
\end{definition}

The homotopy groups of a discrete condensed spectrum are the discrete
condensations of its ordinary homotopy groups:
\[
    \pi_n\left(E^{\mathrm{disc}}\right)
    \cong
    \left(\pi_nE\right)^{\mathrm{disc}}
    \qquad
    (n\in\mathbb{Z}).
\]
In particular,
\[
    \pi_n
    \left(
        I_{\mathbb{Q}/\mathbb{Z}}^{\mathrm{disc}}
    \right)
    \cong
    \left(
        \pi_nI_{\mathbb{Q}/\mathbb{Z}}
    \right)^{\mathrm{disc}}.
\]
See \cite[second remark following Definition~3.3]{DerivedStoneEmbedding}.

\begin{definition}[Eilenberg--Mac Lane condensed spectra]
For every condensed abelian group \(A\in\Cond(\Ab)\), we denote by
\[
    HA\in\Cond(\Sp)
\]
the corresponding Eilenberg--Mac Lane condensed spectrum. It is
characterized by
\[
    \pi_0(HA)\cong A
\]
and
\[
    \pi_n(HA)=0
    \qquad
    \text{for every }n\neq 0.
\]
\end{definition}

For an ordinary abelian group \(M\), there is a canonical equivalence
\[
    H\left(M^{\mathrm{disc}}\right)
    \simeq
    (HM)^{\mathrm{disc}}.
\]
In particular,
\[
    H\left(
        \left(\mathbb{Q}/\mathbb{Z}\right)^{\mathrm{disc}}
    \right)
    \simeq
    \left(
        H(\mathbb{Q}/\mathbb{Z})
    \right)^{\mathrm{disc}}.
\]

\begin{definition}[The symmetric monoidal structure on condensed spectra]

The category \(\Cond(\Sp)\) is a presentable stable closed symmetric
monoidal \(\infty\)-category, with unit given by
\(\mathbb{S}^{\mathrm{disc}}\); see
\cite[Proposition~1.3.4.7]{SAG}.

The closed symmetric monoidal structure on
\(\Cond(\Sp)\) provides an internal Hom condensed spectrum
\[
    \uHom_{\Cond(\Sp)}(X,Y)
    \in\Cond(\Sp),
\]
characterized by the natural equivalence
\[
    \Map_{\Cond(\Sp)}
    \left(
        Z,
        \uHom_{\Cond(\Sp)}(X,Y)
    \right)
    \simeq
    \Map_{\Cond(\Sp)}
    \left(
        Z\otimes X,
        Y
    \right)
\]
for all \(Z\in\Cond(\Sp)\).

\end{definition}

The \((-)^{\mathrm{disc}}:\) functor is strong symmetric monoidal with respect to this described structure.

\begin{notation}

Similarly to the notation used in the abelian and stable settings, we will use
\[
    [X,Y]_{\Cond(\Ab)}:=\uHom_{\mathcal \Cond(\Ab)}(X,Y).
\]
and
\[
    [X,Y]_{\Cond(\Sp)}
    :=
    \pi_0 \uHom_{\Cond(\Sp)}(X,Y)
\]
to denote the condensed abelian group of morphisms from \(X\) to \(Y\).

\end{notation}

\subsection{Locally Compact Condensed Spectra}

We write \(\LCA\) for the category of locally compact abelian groups and
continuous homomorphisms.

\begin{definition}
Define
\[
\Cond(\Sp)_{\mathrm{lc}} \subseteq \Cond(\Sp)
\]
as the full subcategory of \emph{locally compact condensed spectra} spanned by those objects \(X\)
for which \(\pi_nX\) is represented by a locally compact abelian group for
every \(n\in\mathbb{Z}\).
\end{definition}
\begin{corollary}
\label{cor:condab-ext-circle}
Let \(A\) be a locally compact abelian group then
\[
\RHom_{\Cond(\Ab)}(A,\mathbb{T})=\uHom_{\Cond(\Ab)}(A,\mathbb{T})[0].
\]
\end{corollary}
\begin{proof}
\cite[Corollary~2.9]{weilmoore_clausen}
\end{proof}
\subsection{The condensed Pontryagin dualizing spectrum}
\label{sec:condensed-dualizing}

Let
\[
  I_{\mathbb{Q}/\mathbb{Z}}\in\Sp
\]
be the classical Brown--Comenetz spectrum and let
\[
  I_{\mathbb{Q}/\mathbb{Z}}^{\mathrm{disc}}
  \in\Cond(\Sp)
\]
be its discrete condensation.

We first record some useful properties of this discrete condensation.

\begin{lemma}
\label{lem:disc-QZ-injective}
For every profinite set \(S\), let 
\[
    B_S :=(\mathbb{Q}/\mathbb{Z})^{\mathrm{disc}}(S).
\]
Then \(B_S\) is an injective abelian group.
\end{lemma}

\begin{proof}
Recall that an abelian group is injective if and only if it is divisible.
By definition,
\[
    B_S
    =
    \operatorname{Cont}
    \left(S,\mathbb{Q}/\mathbb{Z}\right),
\]
where \(\mathbb{Q}/\mathbb{Z}\) is given the discrete topology. Since
\(\mathbb{Q}/\mathbb{Z}\) is divisible, the group of continuous maps
\(\operatorname{Cont}(S,\mathbb{Q}/\mathbb{Z})\) is divisible. Hence
\(B_S\) is injective.
\end{proof}

\begin{lemma}
\label{lem:discrete-BC-sections}
For every profinite set \(S\), there is a canonical equivalence
\[
    I_{\mathbb{Q}/\mathbb{Z}}^{\mathrm{disc}}(S)
    \simeq
    I_{B_S},
\]
    where \(I_{B_S}\) denotes the Brown--Comenetz spectrum associated to the injective \(B_S\).
\end{lemma}

\begin{proof}
Set
\[
    K_S:=I_{\mathbb{Q}/\mathbb{Z}}^{\mathrm{disc}}(S).
\]
By the formula for discrete condensation, there is a natural
equivalence
\[
    K_S
    \simeq
    \operatorname*{colim}_{S\twoheadrightarrow S_0}
    I_{\mathbb{Q}/\mathbb{Z}}^{S_0},
\]
where \(S_0\) ranges over the finite continuous quotients of \(S\).
Since homotopy groups commute with filtered colimits, this gives
a canonical identification
\[
    \pi_0K_S
    \cong
    \operatorname*{colim}_{S\twoheadrightarrow S_0}
    (\mathbb{Q}/\mathbb{Z})^{S_0}
    \cong B_S.
\]

By \Cref{lem:disc-QZ-injective}, the group \(B_S\) is injective.
Choose a Brown--Comenetz lift \(I_{B_S}\), equipped with an
identification \(\pi_0I_{B_S}\cong B_S\).
Its universal property gives an isomorphism
\[
    [K_S,I_{B_S}]_{\Sp}
    \xrightarrow{\ \sim\ }
    \Hom_{\Ab}(\pi_0K_S,B_S).
\]
Let
\[
    \phi_S:K_S\longrightarrow I_{B_S}
\]
represent the homotopy class corresponding to the canonical
identification \(\pi_0K_S\cong B_S\).

We claim that \(\phi_S\) is an equivalence.
For every finite quotient \(S_0\) and every \(k\in\mathbb Z\),
the Brown--Comenetz universal property gives natural
isomorphisms
\[
    \pi_k\bigl(I_{\mathbb{Q}/\mathbb{Z}}^{S_0}\bigr)
    \cong
    \Hom_{\Ab}
    \bigl(\pi_{-k}\mathbb S,(\mathbb{Q}/\mathbb{Z})^{S_0}\bigr),
\]
and
\[
    \pi_k I_{B_S}
    \cong
    \Hom_{\Ab}(\pi_{-k}\mathbb S,B_S).
\]
The composite
\[
    I_{\mathbb{Q}/\mathbb{Z}}^{S_0}
    \longrightarrow K_S
    \xrightarrow{\phi_S} I_{B_S}
\]
induces on \(\pi_0\) the canonical coefficient map
\[
    (\mathbb{Q}/\mathbb{Z})^{S_0}\longrightarrow B_S.
\]
Naturality of the Brown--Comenetz comparison therefore identifies
\(\pi_k(\phi_S)\) with the canonical map
\[
    \operatorname*{colim}_{S\twoheadrightarrow S_0}
    \Hom_{\Ab}
    \bigl(\pi_{-k}\mathbb S,(\mathbb{Q}/\mathbb{Z})^{S_0}\bigr)
    \longrightarrow
    \Hom_{\Ab}(\pi_{-k}\mathbb S,B_S).
\]
The abelian group \(\pi_{-k}\mathbb S\) is finitely presented,
so this map is an isomorphism for every \(k\in\mathbb Z\).
Thus \(\phi_S\) induces an isomorphism on every homotopy group
and is consequently an equivalence.
\end{proof}

We now turn to the central construction of this paper: the condensed
Pontryagin dualizing spectrum.

\medskip

By \Cref{thm:injective-lifts}, there is a natural isomorphism
\[
  [H(\mathbb{Q}/\mathbb{Z}),I_{\mathbb{Q}/\mathbb{Z}}]_{\Sp}
  \cong
  \End_{\Ab}(\mathbb{Q}/\mathbb{Z}).
\]
Let
\[
  h:H(\mathbb{Q}/\mathbb{Z})
  \longrightarrow I_{\mathbb{Q}/\mathbb{Z}}
\]
denote the morphism corresponding to
\(\id_{\mathbb{Q}/\mathbb{Z}}\), and write \(h^{\mathrm{disc}}\) for its
discrete condensation.

We can now define the desired spectrum.

\begin{definition}[The condensed Pontryagin dualizing spectrum]
\label{def:condensed-pontryagin-dualizing-spectrum}
Define
\[
    I_{\mathbb{T}}\in\Cond(\Sp)
\]
by the following pushout square:
\[\begin{tikzcd}
	{H\left(
	    \left(\mathbb{Q}/\mathbb{Z}\right)^{\mathrm{disc}}
	\right)} & {H\mathbb{T}} \\
	{I_{\mathbb{Q}/\mathbb{Z}}^{\mathrm{disc}}} & {I_{\mathbb{T}}.}
	\arrow["{H\iota}", from=1-1, to=1-2]
	\arrow["{h^{\mathrm{disc}}}"', from=1-1, to=2-1]
	\arrow[from=1-2, to=2-2]
	\arrow[from=2-1, to=2-2]
	\arrow["\lrcorner"{anchor=center, pos=0.125, rotate=180}, draw=none, from=2-2, to=1-1]
\end{tikzcd}\]
\end{definition}

The construction glues the classical Brown--Comenetz spectrum, which
has the desired stable mapping property, to the Eilenberg--Mac Lane
spectrum \(H\mathbb{T}\), which has the desired degree-zero homotopy group.

\begin{proposition}
\label{prop:pi-zero-condensed-pontryagin-dualizing-spectrum}
The following hold for the condensed spectrum \( I_{\mathbb{T}} \):
\begin{enumerate}
    \item \( I_{\mathbb{T}} \in \Cond(\Sp)_{\leq 0} \).
    \item There is a canonical isomorphism \( \pi_0 I_{\mathbb{T}} \cong \mathbb{T} \) in \( \Cond(\Ab) \).
\end{enumerate}
\end{proposition}

\begin{proof}
The defining pushout of \(I_{\mathbb{T}}\) induces a cofiber sequence
\[
    H\left(
        \left(\mathbb{Q}/\mathbb{Z}\right)^{\mathrm{disc}}
    \right)
    \xrightarrow{\,(h^{\mathrm{disc}},-H\iota)\,}
    I_{\mathbb{Q}/\mathbb{Z}}^{\mathrm{disc}}
    \oplus H\mathbb{T}
    \longrightarrow
    I_{\mathbb{T}}.
\]

All terms before \(I_{\mathbb{T}}\) are coconnective, and the morphism on
degree-zero homotopy groups is
\[
    \left(
        \id,
        -\iota
    \right):
    \left(\mathbb{Q}/\mathbb{Z}\right)^{\mathrm{disc}}
    \lhook\joinrel\longrightarrow
    \left(\mathbb{Q}/\mathbb{Z}\right)^{\mathrm{disc}}
    \oplus \mathbb{T},
\]
which is a monomorphism.  Here we have used
\[
    \pi_0
    \left(
        I_{\mathbb{Q}/\mathbb{Z}}^{\mathrm{disc}}
    \right)
    \cong
    \left(\mathbb{Q}/\mathbb{Z}\right)^{\mathrm{disc}},
    \qquad
    \pi_0(h^{\mathrm{disc}})
    =
    \id_{
        \left(\mathbb{Q}/\mathbb{Z}\right)^{\mathrm{disc}}
    }.
\]

Thus the associated long exact sequence gives
\[
    \pi_m I_{\mathbb{T}}=0
    \qquad
    \text{for every }m>0.
\]
Similarly, since
\[
    \pi_{-1}
    H\left(
        \left(\mathbb{Q}/\mathbb{Z}\right)^{\mathrm{disc}}
    \right)
    =0,
\]
We get in degree zero:
\[
\begin{tikzcd}[column sep=large]
\left(\mathbb{Q}/\mathbb{Z}\right)^{\mathrm{disc}}
\arrow[r,"{(\id,-\iota)}"]
&
\left(\mathbb{Q}/\mathbb{Z}\right)^{\mathrm{disc}}
\oplus \mathbb{T}
\arrow[r]
&
\pi_0I_{\mathbb{T}}
\arrow[r]
&
0.
\end{tikzcd}
\].

Consider the morphism in \(\Cond(\Ab)\)
\[
    \varphi:
    \left(\mathbb{Q}/\mathbb{Z}\right)^{\mathrm{disc}}
    \oplus \mathbb{T}
    \longrightarrow
    \mathbb{T}
\]
defined by
\[
    \varphi(a,z)=\iota(a)+z.
\]
It is an epimorphism, since every \(z\in \mathbb{T}\) is the image of
\((0,z)\). Moreover,
\[
    \ker(\varphi)
    =
    \left\{
        \bigl(a,-\iota(a)\bigr)
        \,\middle|\,
        a\in
        \left(\mathbb{Q}/\mathbb{Z}\right)^{\mathrm{disc}}
    \right\},
\]
which is precisely the image of
\[
    (\id,-\iota):
    \left(\mathbb{Q}/\mathbb{Z}\right)^{\mathrm{disc}}
    \longrightarrow
    \left(\mathbb{Q}/\mathbb{Z}\right)^{\mathrm{disc}}
    \oplus \mathbb{T}.
\]
Therefore \(\varphi\) induces a canonical isomorphism
\[
    \coker(\id,-\iota)
    \xrightarrow{\ \sim\ }
    \mathbb{T}.
\]
It follows that
\[
    \pi_0I_{\mathbb{T}}\cong \mathbb{T}.
\]
\end{proof}

\section{Proving the Duality}
In this section we prove the main duality theorem. The dualizing spectrum was constructed in \Cref{sec:condensed-dualizing}; we now establish its universal property and deduce the resulting duality. The key ingredient is the coinduction functor right adjoint to the forgetful functor
\[
D(\Cond(\Ab))\longrightarrow \Cond(\Sp).
\]
We begin by introducing this coinduction functor and establishing several useful lemmas concerning it. We then formulate and prove the universal property of the dualizing spectrum, and finally use this universal property to deduce the desired duality.

\subsection{The Restriction--Coinduction Adjunction}
There is a canonical \(t\)-exact
equivalence
\[
  D(\Cond(\Ab))
  \simeq
  \Mod_{H(\mathbb{Z}^{\mathrm{disc}})}(\Cond(\Sp)).
\]
Under this equivalence, the derived tensor product and derived internal Hom
agree with the relative tensor product and internal Hom of
\(H(\mathbb{Z}^{\mathrm{disc}})\)-modules.

Consider the forgetful functor
\[
    U:
    D\bigl(\Cond(\Ab)\bigr)
    \simeq
    \Mod_{H(\mathbb{Z}^{\mathrm{disc}})}
    \bigl(\Cond(\Sp)\bigr)
    \longrightarrow
    \Cond(\Sp),
\]
and let
\[
    R(Y)
    :=
    \uHom_{\Cond(\Sp)}
    \bigl(H(\mathbb{Z}^{\mathrm{disc}}),Y\bigr)
\]
be its right adjoint.

\begin{lemma}[Enriched restriction--coinduction]
\label{lem:enriched-restriction-coinduction}
For \(M\in D(\Cond(\Ab))\) and \(Y\in\Cond(\Sp)\), there is a natural
equivalence
\[
  \uHom_{\Cond(\Sp)}(U(M),Y)
  \simeq
  U\left(\uHom_{D(\Cond(\Ab))}(M,R(Y))\right).
\]
\end{lemma}

\begin{proof}
Let \(Z\in\Cond(\Sp)\).  The free
\(H(\mathbb{Z}^{\mathrm{disc}})\)-module on \(Z\) is
\(H(\mathbb{Z}^{\mathrm{disc}})\otimes Z\). Thus,
\[
\begin{aligned}
&\Map_{\Cond(\Sp)}
\left(Z,U\left(\uHom_{D(\Cond(\Ab))}(M,R(Y))\right)\right)\\
&\quad\simeq
\Map_{D(\Cond(\Ab))}
\left(H(\mathbb{Z}^{\mathrm{disc}})\otimes Z,
      \uHom_{D(\Cond(\Ab))}(M,R(Y)) \right)\\
&\quad\simeq
\Map_{D(\Cond(\Ab))}
\left((H(\mathbb{Z}^{\mathrm{disc}})\otimes Z)
      \otimes_{H(\mathbb{Z}^{\mathrm{disc}})}M,
      R(Y)\right)\\
&\quad\simeq
\Map_{\Cond(\Sp)}(Z\otimes U(M),Y)\\
&\quad\simeq
\Map_{\Cond(\Sp)}
\left(Z,\uHom_{\Cond(\Sp)}(U(M),Y)\right).
\end{aligned}
\]
The equivalence is natural in \(Z\), and hence the claim follows by Yoneda.

\end{proof}

\begin{lemma}
\label{lem:coinduction-brown-comenetz}
Let
\[
  h:H(\mathbb{Q}/\mathbb{Z})
  \longrightarrow I_{\mathbb{Q}/\mathbb{Z}}
\]
be the morphism corresponding to
\(\id_{\mathbb{Q}/\mathbb{Z}}\), and let \(h^{\mathrm{disc}}\) denote its
discrete condensation. Then the adjoint of \(h^{\mathrm{disc}}\),
\[
    \widetilde{h}:
    H\left(
        \left(\mathbb{Q}/\mathbb{Z}\right)^{\mathrm{disc}}
    \right)
    \longrightarrow
    R\left(
        I_{\mathbb{Q}/\mathbb{Z}}^{\mathrm{disc}}
    \right),
\]
is an equivalence in \(D\bigl(\Cond(\Ab)\bigr)\).
\end{lemma}

\begin{proof}
We will show that \(\widetilde{h}\) is an equivalence by checking it pointwise. Let \(S\) be a profinite set and define
\[
    B_S := (\mathbb{Q}/\mathbb{Z})^{\mathrm{disc}}(S) = \operatorname*{colim}_{S\twoheadrightarrow S_0} (\mathbb{Q}/\mathbb{Z})^{S_0}.
\]
By \Cref{lem:discrete-BC-sections}, we have \(I_{\mathbb{Q}/\mathbb{Z}}^{\mathrm{disc}}(S) \simeq I_{B_S}\). Combining this with the analogous identification \(H(\mathbb{Q}/\mathbb{Z})^{\mathrm{disc}}(S) \simeq H B_S\), the evaluation of \(h^{\mathrm{disc}}\) at \(S\) yields a map
\[
    h^{\mathrm{disc}}(S): H B_S \longrightarrow I_{B_S}.
\]
We claim this map is precisely the Brown--Comenetz morphism associated with \(\id_{B_S}\). By definition of the discrete condensation, \(h^{\mathrm{disc}}(S)\) is given by the colimit \(\operatorname*{colim}_{S\twoheadrightarrow S_0} h^{S_0}\). Taking the induced map on \(\pi_0\), we obtain the commutative diagram
\[
\begin{tikzcd}[column sep=large]
B_S 
    \arrow[r, "\sim"] 
    \arrow[d, "\pi_0(h^{\mathrm{disc}}(S))"'] 
& 
\displaystyle \operatorname*{colim}_{S\twoheadrightarrow S_0} (\mathbb{Q}/\mathbb{Z})^{S_0} 
    \arrow[d, "\operatorname*{colim}(\id_{\mathbb{Q}/\mathbb{Z}})^{S_0}"] 
\\
B_S 
& 
\displaystyle \operatorname*{colim}_{S\twoheadrightarrow S_0} (\mathbb{Q}/\mathbb{Z})^{S_0} 
    \arrow[l, "\sim"']
\end{tikzcd}
\]
which confirms that \(\pi_0(h^{\mathrm{disc}}(S))\) is the identity on \(B_S\), proving the claim.

Next, we analyze the target of \(\widetilde{h}(S)\). We have
\[
    R\left( I_{\mathbb{Q}/\mathbb{Z}}^{\mathrm{disc}} \right)(S) 
    \simeq \underline{\operatorname{Hom}}_{\Sp} \left( H\mathbb{Z}, I_{\mathbb{Q}/\mathbb{Z}}^{\mathrm{disc}}(S) \right) 
    \simeq \underline{\operatorname{Hom}}_{\Sp} \left( H\mathbb{Z}, I_{B_S} \right).
\]
The homotopy groups of this mapping spectrum are given by
\[
\begin{aligned}
    \pi_k \underline{\operatorname{Hom}}_{\Sp} \left(H\mathbb{Z}, I_{B_S}\right) 
    &\cong \left[ \Sigma^k H\mathbb{Z}, I_{B_S} \right]_{\Sp} \\
    &\cong \operatorname{Hom}_{\Ab} \left( \pi_{-k}H\mathbb{Z}, B_S \right) \\
    &\cong 
    \begin{cases}
        B_S, & k=0,\\
        0,   & k\neq0.
    \end{cases}
\end{aligned}
\]
It follows that \(\underline{\operatorname{Hom}}_{\Sp}(H\mathbb{Z}, I_{B_S}) \simeq HB_S\) as an \(H\mathbb{Z}\)-module. 

Under these identifications, the pointwise adjoint
\[
    \widetilde{h}(S): HB_S \longrightarrow \underline{\operatorname{Hom}}_{\Sp} \left(H\mathbb{Z}, I_{B_S}\right) \simeq HB_S
\]
is determined by the unit \(\eta_{HB_S}\) of the adjunction via \(\widetilde{h}(S) = h^{\mathrm{disc}}(S)_\ast \circ \eta_{HB_S}\). Evaluating this on \(\pi_0\) yields the commutative diagram
\[\begin{tikzcd}[column sep=large]
    {\pi_0(HB_S)} 
        \arrow[r, "\pi_0(\eta_{HB_S})"] 
        \arrow[d, equal] 
    & {[H\mathbb{Z}, HB_S]_{\Sp}} 
        \arrow[r, "\pi_0(h^{\mathrm{disc}}(S))_\ast"] 
        \arrow[d, "\sim", "\rho_{H\mathbb{Z}, HB_S}"'] 
    & {[H\mathbb{Z}, I_{B_S}]_{\Sp}} 
        \arrow[d, "\sim"',"\rho_{H\mathbb{Z}, I_{B_S}}"] 
    \\
    {B_S} 
        \arrow[r, equal] 
    & {[\mathbb{Z}, B_S]_{\Ab}} 
        \arrow[r, "\sim","(\id_{B_S})_\ast"'] 
    & {[\mathbb{Z}, B_S]_{\Ab}}
\end{tikzcd}\]
which demonstrates that \(\pi_0(\widetilde{h}(S)) = \id_{B_S}\).

Because both the source and target spectra are Eilenberg--MacLane spectra (and are thus concentrated entirely in degree zero), the fact that \(\widetilde{h}(S)\) induces an isomorphism on \(\pi_0\) guarantees that it is an equivalence of spectra.

Since this pointwise equivalence holds naturally for all profinite sets \(S\), \(\widetilde{h}\) is an equivalence. 
\end{proof}

\begin{lemma}
\label{lem:coinduction-of-condensed-dualizing-spectrum}

There is a canonical equivalence
\[
    R(I_{\mathbb{T}})\simeq H\mathbb{T}
\]

in \(D\bigl(\Cond(\Ab)\bigr)\).
\end{lemma}

\begin{proof}
Let
\[
    V
    :=
    \coker_{\Cond(\Ab)}
    \left(
        \left(\mathbb{Q}/\mathbb{Z}\right)^{\mathrm{disc}}
        \longrightarrow
        \mathbb{T}
    \right).
\]

Thus there is a short exact sequence in
\(\Cond(\Ab)\)
\[
    0
    \longrightarrow
    \left(\mathbb{Q}/\mathbb{Z}\right)^{\mathrm{disc}}
    \longrightarrow
    \mathbb{T}
    \longrightarrow
    V
    \longrightarrow
    0,
\]
and hence a cofiber sequence
\[
    H\left(
        \left(\mathbb{Q}/\mathbb{Z}\right)^{\mathrm{disc}}
    \right)
    \longrightarrow
    H\mathbb{T}
    \longrightarrow
    HV
\]
in \(D(\Cond(\Ab))\), and therefore also in
\(\Cond(\Sp)\) after applying the forgetful functor.

Now consider the defining pushout square of \(I_{\mathbb{T}}\). By the
pasting law for pushouts, together with the above cofiber sequence, we
obtain the diagram
\[
\begin{tikzcd}
    {H\left((\mathbb{Q}/\mathbb{Z})^{\mathrm{disc}}\right)}
        & {H\mathbb{T}} \\
    {I_{\mathbb{Q}/\mathbb{Z}}^{\mathrm{disc}}}
        & {I_{\mathbb{T}}} \\
    0 & HV
    \arrow["{H\iota}", from=1-1, to=1-2]
    \arrow["{h^{\mathrm{disc}}}"', from=1-1, to=2-1]
    \arrow[from=1-2, to=2-2]
    \arrow[from=2-1, to=2-2]
    \arrow[from=2-1, to=3-1]
    \arrow["\lrcorner"{anchor=center, pos=0.125, rotate=180},
        draw=none, from=2-2, to=1-1]
    \arrow[from=2-2, to=3-2]
    \arrow[from=3-1, to=3-2]
    \arrow["\lrcorner"{anchor=center, pos=0.125, rotate=180},
        draw=none, from=3-2, to=2-1]
\end{tikzcd}
\]
where both squares are pushouts. In particular, we obtain a cofiber
sequence
\[
    I_{\mathbb{Q}/\mathbb{Z}}^{\mathrm{disc}}
    \longrightarrow
    I_{\mathbb{T}}
    \longrightarrow
    HV
\]
in \(\Cond(\Sp)\).

We can therefore arrange the two cofiber sequences into the following
commutative diagram in \(\Cond(\Sp)\):
\[
\begin{tikzcd}
    {H\left((\mathbb{Q}/\mathbb{Z})^{\mathrm{disc}}\right)}
        & {H\mathbb{T}} & HV \\
    {I_{\mathbb{Q}/\mathbb{Z}}^{\mathrm{disc}}}
        & {I_{\mathbb{T}}} & HV
    \arrow["{H\iota}", from=1-1, to=1-2]
    \arrow["{h^{\mathrm{disc}}}"', from=1-1, to=2-1]
    \arrow[from=1-2, to=1-3]
    \arrow[from=1-2, to=2-2]
    \arrow[equal, from=1-3, to=2-3]
    \arrow[from=2-1, to=2-2]
    \arrow[from=2-2, to=2-3]
\end{tikzcd}
\]

Applying the adjunction \(U\dashv R\) gives the corresponding diagram
in \(D(\Cond(\Ab))\):
\[
\begin{tikzcd}[column sep=large,row sep=large]
H\left(
    \left(\mathbb{Q}/\mathbb{Z}\right)^{\mathrm{disc}}
\right)
\arrow[r]
\arrow[d,"\widetilde{h}"']
&
H\mathbb{T}
\arrow[r]
\arrow[d]
&
HV
\arrow[d,"\eta_{HV}"]
\\
R\left(
    I_{\mathbb{Q}/\mathbb{Z}}^{\mathrm{disc}}
\right)
\arrow[r]
&
R(I_{\mathbb{T}})
\arrow[r]
&
R(HV).
\end{tikzcd}
\]
Since \(R\) is a right adjoint between stable \(\infty\)-categories, it
is exact. Hence the bottom row is again a cofiber sequence.

The left vertical map is an equivalence by
\Cref{lem:coinduction-brown-comenetz}. We will show below that
\(\eta_{HV}\) is also an equivalence. Granting this for the moment,
the two-out-of-three property for morphisms of cofiber sequences gives
\[
    H\mathbb{T}
    \xrightarrow{\ \sim\ }
    R(I_{\mathbb{T}}).
\]
It remains to prove that \(\eta_{HV}\) is an equivalence.

We first identify \(V\) as a cokernel in condensed abelian groups:
\[
    V
    \cong
    \coker_{\Cond(\Ab)}
    \bigl(\mathbb{Q}^{\mathrm{disc}}\longrightarrow\mathbb{R}\bigr).
\]
Indeed, quotienting by the common subgroup
\(\mathbb{Z}^{\mathrm{disc}}\) identifies the defining map with
\[
    (\mathbb{Q}/\mathbb{Z})^{\mathrm{disc}}
    \longrightarrow \mathbb{T}.
\]
Both \(\mathbb{Q}^{\mathrm{disc}}\) and \(\mathbb{R}\) are condensed
\(\mathbb{Q}\)-vector spaces, and the inclusion is
\(\mathbb{Q}\)-linear. Hence \(V\) is a condensed
\(\mathbb{Q}\)-vector space, so \(HV\) is a rational condensed
spectrum.

Put
\[
    \mathcal C:=\Cond(\Sp),
    \qquad
    A_0:=H(\mathbb{Z}^{\mathrm{disc}}).
\]
The unit
\[
    u:\mathbb{S}^{\mathrm{disc}}\longrightarrow A_0
\]
is a rational equivalence, since
\(\mathbb{S}_{\mathbb{Q}}\simeq H\mathbb{Q}\).
Rationalization is smashing, so internal Hom into a rational
object inverts rational equivalences in its first variable.
Consequently, precomposition with \(u\) gives an equivalence
\[
\begin{aligned}
    e:\ U R(HV)
    &=
    \uHom_{\mathcal C}(A_0,HV)\\
    &\xrightarrow{\ u^*\ }
    \uHom_{\mathcal C}(\mathbb{S}^{\mathrm{disc}},HV)
    \simeq HV.
\end{aligned}
\]

The underlying map of the adjunction unit
\(\eta_{HV}\) is adjoint to the \(A_0\)-module action on \(HV\).
Unitality of this action therefore gives
\[
    e\circ U(\eta_{HV})=\id_{HV}.
\]
Since \(e\) is an equivalence, \(U(\eta_{HV})\) is its inverse.
The forgetful functor \(U\) is conservative, so \(\eta_{HV}\)
is an equivalence. This completes the proof.
\end{proof}

\subsection{The universal property of the dualizing spectrum}

We can now state the universal property of the dualizing spectrum constructed above, which captures the main properties we sought.

\begin{notation}
For a condensed spectrum \(Y\), write
\[
    \widehat{Y}
    :=
    \uHom_{\Cond(\Sp)}
    \left(Y,I_{\mathbb{T}}\right).
\].
For a condensed abelian group \(B\), write
\[
    \widehat{B}
    :=
    \uHom_{\Cond(\Ab)}
    \left(B,\mathbb{T}\right).
\].

\end{notation}

\begin{proposition}[The universal property of the dualizing spectrum]
\label{prop:univ-prop-dualizing-spec}
Let
\[
    X\in\Cond(\Sp)_{\mathrm{lc}}
\]
be a locally compact condensed spectrum. Then the natural morphism of
condensed abelian groups
\[
    \rho_X:
    \pi_0\widehat{X}
    \longrightarrow
    \widehat{\pi_0 X}
\]
is an isomorphism. Here \(\rho_X\) is adjoint to the pairing
\[
\begin{aligned}
  \pi_0\uHom_{\Cond(\Sp)}(X,I_{\mathbb{T}})\otimes\pi_0X
  &\longrightarrow
  \pi_0\left(
    \uHom_{\Cond(\Sp)}(X,I_{\mathbb{T}})\otimes X
  \right)\\
  &\xrightarrow{\ \pi_0(\operatorname{ev}_X)\ }
  \pi_0I_{\mathbb{T}}
  \xrightarrow{\ \sim\ }\mathbb{T}.
\end{aligned}
\]
\end{proposition}

We first treat the case where \(X\) is an Eilenberg--Mac Lane spectrum.
\begin{lemma}
\label{lem:eilenberg-maclane-calculation}
For every locally compact abelian group \(A\), there is a natural
equivalence
\begin{equation}
\label{eq:dual-eilenberg-maclane}
    \widehat{HA}\simeq H\widehat A
\end{equation}
whose induced identification on \(\pi_0\) is \(\rho_{HA}\).
Consequently, for every \(n\in\mathbb Z\),
\[
    \widehat{\Sigma^nHA}\simeq\Sigma^{-n}H\widehat A.
\]
\end{lemma}

\begin{proof}
Let \(j:H\mathbb T\to I_{\mathbb T}\) be the canonical pushout map.
By the proof of
\Cref{lem:coinduction-of-condensed-dualizing-spectrum}, its adjunct
\[
    \alpha:H\mathbb T\xrightarrow{\sim}R(I_{\mathbb T})
\]
is an equivalence.
Enriched restriction--coinduction, followed by \(\alpha^{-1}\),
therefore gives
\[
\begin{aligned}
    \widehat{HA}
    &\simeq
    U\bigl(\uHom_{D(\Cond(\Ab))}(A[0],R(I_{\mathbb T}))\bigr)\\
    &\simeq
    U\bigl(\mathbf R\uHom_{\Cond(\Ab)}(A,\mathbb T)\bigr)\\
    &\simeq H\widehat A,
\end{aligned}
\]
where the last equivalence is \Cref{cor:condab-ext-circle}.

Under this equivalence, evaluation is the composite
\[
    H\widehat A\otimes HA
    \longrightarrow H\mathbb T
    \xrightarrow{j}I_{\mathbb T}.
\]
The first arrow is induced by derived evaluation, using the
canonical map from the ambient tensor product to the relative
module tensor product. Since \(\pi_0(j)=\id_{\mathbb T}\), the
induced comparison on \(\pi_0\) is precisely \(\rho_{HA}\).
\end{proof}

We now prove the proposition for a general locally compact condensed
spectrum \(X\).

\begin{proof}[Proof of \Cref{prop:univ-prop-dualizing-spec}]
Since \(I_{\mathbb T}\le0\), it follows that
\(\widehat{\tau_{\ge1}X}\le-1\). Dualizing the cofiber sequence
\[
    \tau_{\ge1}X\longrightarrow X\longrightarrow\tau_{\le0}X
\]
therefore gives an isomorphism
\[
    \pi_0\widehat{\tau_{\le0}X}\xrightarrow{\sim}\pi_0\widehat X.
\]
By naturality of evaluation, it suffices to treat \(X\le0\).

For \(N\ge0\), set \(X_N:=\tau_{\ge-N}X\). Then
\(X_0\simeq H\pi_0X\), and there are cofiber sequences
\[
    X_N\longrightarrow X_{N+1}
    \longrightarrow\Sigma^{-N-1}H\pi_{-N-1}X.
\]
By \Cref{lem:eilenberg-maclane-calculation}, their duals are
fiber sequences
\[
    \Sigma^{N+1}H\widehat{\pi_{-N-1}X}
    \longrightarrow\widehat{X_{N+1}}
    \longrightarrow\widehat{X_N}.
\]
Consequently, the maps
\(\pi_0\widehat{X_{N+1}}\to\pi_0\widehat{X_N}\) are isomorphisms
for \(N\ge0\), and the corresponding maps on \(\pi_1\) are
isomorphisms for \(N\ge1\).
The \(\pi_0\)-system is thus canonically constant with value
\[
    \pi_0\widehat{X_0}\cong\widehat{\pi_0X},
\]
and the \(\pi_1\)-system is eventually constant.

Filtered colimits commute with homotopy objects, and homotopy
objects detect equivalences. Hence the canonical map gives
\[
    X\simeq\operatorname*{colim}_{N\ge0}X_N,
    \qquad
    \widehat X\simeq\varprojlim_N\widehat{X_N}.
\]
For an extremally disconnected profinite set \(S\), evaluation
is \(t\)-exact and preserves limits. The Milnor exact sequence
of spectra therefore gives
\[
\begin{aligned}
    0&\longrightarrow
    \varprojlim_N^1\pi_1\bigl(\widehat{X_N}(S)\bigr)
    \longrightarrow\pi_0\bigl(\widehat X(S)\bigr)\\
    &\longrightarrow
    \varprojlim_N\pi_0\bigl(\widehat{X_N}(S)\bigr)
    \longrightarrow0.
\end{aligned}
\]
The first term vanishes by eventual constancy. It follows that
restriction to \(X_0\) induces an isomorphism
\[
    \pi_0\widehat X\xrightarrow{\sim}\pi_0\widehat{X_0},
\]
since evaluations on extremally disconnected test spaces detect
isomorphisms of condensed abelian groups.
Composing with \(\rho_{X_0}\) gives
\[
    \pi_0\widehat X\xrightarrow{\sim}\widehat{\pi_0X}.
\]
By naturality of evaluation and the evaluation compatibility in
\Cref{lem:eilenberg-maclane-calculation}, this composite is
\(\rho_X\). This proves the proposition.
\end{proof}

\begin{corollary}[The homotopy formula]
\label{cor:condensed-brown-comenetz-homotopy-formula}
For every \(X\in\Cond(\Sp)_{\mathrm{lc}}\) and every \(n\in\mathbb{Z}\),
Proposition~\ref{prop:univ-prop-dualizing-spec}, applied to \(\Sigma^nX\), gives a natural isomorphism
\[
  \rho_{X,n}:
  \pi_n\widehat X
  \xrightarrow{\ \sim\ }
  \widehat{\pi_{-n}X},
\]
In particular, \(\widehat X\in\Cond(\Sp)_{\mathrm{lc}}\).
\end{corollary}

\subsection{Compatibility with biduality}
\begin{definition}[Pontryagin and condensed Brown--Comenetz duality]
For a condensed abelian group \(A\), define
\[
    \widehat{A}
    :=
    \uHom_{\Cond(\Ab)}
    \left(A,\mathbb{T}\right).
\]
Let
\[
    \operatorname{ev}_A:
    \widehat{A}\otimes A
    \longrightarrow
    \mathbb{T}
\]
be the evaluation morphism. The Pontryagin biduality morphism
\[
    d_A:
    A
    \longrightarrow
    \widehat{\widehat{A}}
\]
is the adjunct of
\[
    A\otimes\widehat{A}
    \xrightarrow{\ \sigma_{A,\widehat{A}}\ }
    \widehat{A}\otimes A
    \xrightarrow{\ \operatorname{ev}_A\ }
    \mathbb{T},
\]
where \(\sigma\) denotes the symmetry.

For a condensed spectrum \(X\), define
\[
    \widehat{X}
    :=
    \uHom_{\Cond(\Sp)}
    \left(X,I_{\mathbb{T}}\right).
\]
Let
\[
    \operatorname{ev}_X:
    \widehat{X}\otimes X
    \longrightarrow
    I_{\mathbb{T}}
\]
be the evaluation morphism. The condensed Brown--Comenetz biduality
morphism
\[
    d_X:
    X
    \longrightarrow
    \widehat{\widehat{X}}
\]
is the adjunct of
\[
    X\otimes\widehat{X}
    \xrightarrow{\ \sigma_{X,\widehat{X}}\ }
    \widehat{X}\otimes X
    \xrightarrow{\ \operatorname{ev}_X\ }
    I_{\mathbb{T}}.
\]
\end{definition}

\begin{theorem}[Compatibility with biduality]
\label{thm:dualization_spec}
Let \(X\in\Cond(Sp)_{\mathrm{lc}}\).  For every \(n\in\mathbb{Z}\), the following
diagram of condensed abelian groups commutes:
\[
\begin{tikzcd}[column sep=large,row sep=large]
\pi_nX
  \arrow[r,"\pi_n(d_X)"]
  \arrow[d,"{(-1)^n d_{\pi_nX}}"']
& \pi_n\widehat{\widehat X}
  \arrow[d,"\rho_{\widehat X,n}"]\\
\widehat{\widehat{\pi_nX}}
  \arrow[r,"{\widehat{\rho_{X,-n}}}"']
& \widehat{\pi_{-n}\widehat X}.
\end{tikzcd}
\]
\end{theorem}

\begin{proof}

Let
\[
  \lambda_{n,X}:
  \pi_{-n}\widehat X\otimes\pi_nX\longrightarrow \mathbb{T}
\]
be the pairing induced by
\(\operatorname{ev}_X:\widehat X\otimes X\to I_{\mathbb{T}}\).  Explicitly,
\[
  \lambda_{n,X}
  =
  \left(
  \pi_{-n}\widehat X\otimes\pi_nX
  \longrightarrow
  \pi_0(\widehat X\otimes X)
  \xrightarrow{\pi_0(\operatorname{ev}_X)}
  \pi_0I_{\mathbb{T}}\cong \mathbb{T}
  \right).
\]
By the construction in \Cref{prop:univ-prop-dualizing-spec}, the comparison
\[
  \rho_{X,-n}:\pi_{-n}\widehat X\longrightarrow\widehat{\pi_nX}
\]
is the adjunct of \(\lambda_{n,X}\).

We compare the two composites in the theorem after applying the tensor--Hom
adjunction
\[
\begin{aligned}
&\Hom_{\Cond(\Ab)}
(\pi_nX,\widehat{\pi_{-n}\widehat X})
\cong
\Hom_{\Cond(\Ab)}
(\pi_nX\otimes\pi_{-n}\widehat X,\mathbb{T}).
\end{aligned}
\]
Recalling that $\widehat{\rho_{X,-n}}=(\rho_{X,-n})^*$, by the naturality of the middle argument in the tensor-hom adjunction we obtain a factorization of the adjunct of
\(\widehat{\rho_{X,-n}}\circ d_{\pi_nX}\) into

\[\begin{tikzcd}
	{\pi_nX \otimes \pi_{-n}\widehat{X}} && \\
	{\pi_nX \otimes \widehat{\pi_{n}X}} && {\mathbb{T}}
	\arrow["{\id\otimes\rho_{X,-n}}"', from=1-1, to=2-1]
	\arrow[from=1-1, to=2-3]
	\arrow["{ev_{\pi_nX}\circ\sigma}"', from=2-1, to=2-3]
\end{tikzcd}\]

where $ev_{\pi_nX}\circ\sigma$ is exactly the adjunct of $d_{\pi_nX}$. By commuting $\sigma$ with $\id\otimes \rho_{X,-n}$ we obtain:

\[
  \pi_nX\otimes\pi_{-n}\widehat X
  \xrightarrow{\sigma}
  \pi_{-n}\widehat X\otimes\pi_nX
  \xrightarrow{\rho_{X,-n}\otimes\id}
  \widehat{\pi_nX}\otimes\pi_nX
  \xrightarrow{ev_{\pi_nX}}\mathbb{T}
\]

And finally, because $ev_{\pi_nX}$ is the counit of $-\otimes\pi_nX\dashv \uHom_{\Cond(\Ab)}(\pi_nX,-)$:

\[
    ev_{\pi_nX}\circ(\rho_{X,-n}\otimes\id_{\pi_nX})=\lambda_{n,X}
\]

so the adjunct of \(\widehat{\rho_{X,-n}}\circ d_{\pi_nX}\) is:

\[
  \pi_nX\otimes\pi_{-n}\widehat X
  \xrightarrow{\sigma}
  \pi_{-n}\widehat X\otimes\pi_nX
  \xrightarrow{\lambda_{n,X}}\mathbb{T}
\]

For the other composite, the adjunct of $\rho_{\widehat{X},n}\circ\pi_n(d_X)$ is:

\[\begin{tikzcd}
	{\pi_n X\otimes\pi_{-n}\widehat{X}} && {\pi_n\widehat{\widehat{X}} \otimes\pi_{-n}\widehat{X}} \\
	&& {\mathbb{T}}
	\arrow["{\pi_n(d_X)\otimes\id}", from=1-1, to=1-3]
	\arrow["{\lambda_{\widehat{X},-n} }", from=1-3, to=2-3]
\end{tikzcd}\]

where
\[
  \lambda_{\widehat{X},-n}  
  =
  \left(
  \pi_{n}\widehat{\widehat{X}}\otimes\pi_{-n}\widehat{X}
  \longrightarrow
  \pi_0(\widehat{\widehat{X}}\otimes \widehat{X})
  \xrightarrow{\pi_0(\operatorname{ev}_{\widehat{X}})}
  \pi_0I_{\mathbb{T}}\cong \mathbb{T}
  \right)
\]

is the left adjunct of $\rho_{\widehat{X},n}$.

We can expand $\lambda_{\widehat{X},-n} $, and by the naturality of the kunneth map, get the following commutative diagram:

\[\begin{tikzcd}
	{\pi_nX\otimes\pi_{-n}\widehat{X}} && { \pi_n\widehat{\widehat{X}} \otimes\pi_{-n}\widehat{X}} & \\
	{\pi_0(X\otimes \widehat{X})} && {\pi_0(\widehat{\widehat{X}}\otimes \widehat{X})} & {\mathbb{T}}
	\arrow["{\pi_n(d_X)\otimes\id}", from=1-1, to=1-3]
	\arrow[from=1-1, to=2-1]
	\arrow[from=1-3, to=2-3]
	\arrow["{\pi_0(d_X\otimes\id_{\widehat{X}})}"', from=2-1, to=2-3]
	\arrow["{\pi_0(ev_{\widehat{X}})}"', from=2-3, to=2-4]
\end{tikzcd}\]

The bottom row is exactly $\pi_0$ of the left adjunct of \(d_X\) in $\Cond(\Sp)$, which by definition is exactly:

\[
    X\otimes\widehat{X}
    \xrightarrow{\ \sigma_{X,\widehat{X}}\ }
    \widehat{X}\otimes X
    \xrightarrow{\ \operatorname{ev}_X\ }
    I_{\mathbb{T}}.
\].

Expanding the formula for the adjoint of $\rho_{\widehat{X},n}\circ\pi_n(d_X)$ we obtain:

\[
    \pi_nX\otimes\pi_{-n}\widehat{X}
    \longrightarrow
    \pi_0(X\otimes\widehat{X})
    \xrightarrow{\pi_0(\sigma_{X,\widehat{X}})}
    \pi_0(\widehat{X}\otimes X)
    \xrightarrow{\pi_0(\operatorname{ev}_X)}
    I_{\mathbb{T}}.
\].

Finally, the Kunneth map can be commuted past the swap map $\pi_0(\sigma_{X,\widehat{X}})$, while introducing the Koszul sign:
\[
  (-1)^{n(-n)}=(-1)^{n^2}=(-1)^n.
\]

Hence we obtain:

\[
    \pi_nX\otimes\pi_{-n}\widehat{X}
    \xrightarrow{(-1)^n\cdot\sigma_{\pi_nX,\pi_{-n}\widehat{X}}}
    \pi_{-n}\widehat{X}\otimes\pi_nX
    \longrightarrow
    \pi_0(\widehat{X}\otimes X)
    \xrightarrow{\pi_0(\operatorname{ev}_X)}
    I_{\mathbb{T}}.
\]

The last two maps are exactly $\lambda_{n,X}$, so this is precisely the adjoint of $\widehat{\rho_{X,-n}}\circ d_{\pi_nX}$, multiplied by the Koszul sign.

Since tensor--Hom adjunction is an isomorphism on Hom-sets, the two morphisms
are equal.  This proves the commutativity of the diagram.
\end{proof}

\begin{theorem}[Condensed Brown--Comenetz--Pontryagin duality]
\label{thm:condensed-brown-comenetz-duality}
For every locally compact condensed spectrum
\[
    X\in\Cond(\Sp)_{\mathrm{lc}},
\]
the canonical biduality morphism
\[
    d_X:
    X
    \longrightarrow
    \widehat{\widehat{X}}
\]
is an equivalence.

Moreover, condensed Brown--Comenetz dualization restricts to a
contravariant equivalence
\[
    \widehat{(-)}:
    \Cond(\Sp)_{\mathrm{lc}}^{\mathrm{op}}
    \xrightarrow{\ \sim\ }
    \Cond(\Sp)_{\mathrm{lc}},
\]
whose quasi-inverse is itself. In particular, it defines an
involutive duality on the category of locally compact condensed
spectra.
\end{theorem}

\begin{proof}
As established in \Cref{cor:condensed-brown-comenetz-homotopy-formula}, Brown--Comenetz dualization preserves locally compact condensed spectra.

By \Cref{thm:dualization_spec}, the following diagram commutes:
\[
\begin{tikzcd}[column sep=large,row sep=large]
\pi_nX
\arrow[r,"\pi_n(d_X)"]
\arrow[d,"{(-1)^n d_{\pi_nX}}"']
&
\pi_n\widehat{\widehat{X}}
\arrow[d,"\rho_{\widehat{X},n}"]
\\
\widehat{\widehat{\pi_nX}}
\arrow[r,"{\widehat{\rho_{X,-n}}}"']
&
\widehat{\pi_{-n}\widehat{X}}.
\end{tikzcd}
\]

We examine the remaining morphisms in the diagram. By
\Cref{cor:condensed-brown-comenetz-homotopy-formula}, the morphism
\[
    \rho_{\widehat{X},n}:
    \pi_n\widehat{\widehat{X}}
    \xrightarrow{\ \sim\ }
    \widehat{\pi_{-n}\widehat{X}}
\]
is an isomorphism. Likewise,
\[
    \rho_{X,-n}:
    \pi_{-n}\widehat{X}
    \xrightarrow{\ \sim\ }
    \widehat{\pi_nX}
\]
is an isomorphism, and hence its dual
\[
    \widehat{\rho_{X,-n}}:
    \widehat{\widehat{\pi_nX}}
    \xrightarrow{\ \sim\ }
    \widehat{\pi_{-n}\widehat{X}}
\]
is an isomorphism.

Finally, since \(\pi_nX\) is a locally compact abelian group, the
Pontryagin--van Kampen duality theorem implies that its biduality
morphism
\[
    d_{\pi_nX}:
    \pi_nX
    \xrightarrow{\ \sim\ }
    \widehat{\widehat{\pi_nX}}
\]
is an isomorphism. Multiplication by \((-1)^n\) is an automorphism, so
\[
    (-1)^n d_{\pi_nX}
\]
is also an isomorphism.

Consequently, the commutative diagram gives
\[
    \pi_n(d_X)
    =
    \rho_{\widehat{X},n}^{-1}
    \circ
    \widehat{\rho_{X,-n}}
    \circ
    (-1)^n d_{\pi_nX}.
\]
It follows that
\[
    \pi_n(d_X):
    \pi_nX
    \xrightarrow{\ \sim\ }
    \pi_n\widehat{\widehat{X}}
\]
is an isomorphism for every \(n\in\mathbb{Z}\).

A morphism of condensed spectra is an equivalence if and
only if it induces an isomorphism on every homotopy group. Hence
\[
    d_X:
    X
    \xrightarrow{\ \sim\ }
    \widehat{\widehat{X}}
\]
is an equivalence.

We thus obtain a contravariant functor
\[
    \widehat{(-)}:
    \Cond(\Sp)_{\mathrm{lc}}^{\mathrm{op}}
    \longrightarrow
    \Cond(\Sp)_{\mathrm{lc}}
\]
together with a natural equivalence
\[
    d:
    \id_{\Cond(\Sp)_{\mathrm{lc}}}
    \xrightarrow{\ \sim\ }
    \widehat{\widehat{(-)}}.
\]
It follows that \(\widehat{(-)}\) is an equivalence whose
quasi-inverse is itself.
\end{proof}
\begin{corollary}[Biduality for ordinary spectra]
\label{cor:biduality-ordinary-spectra}
For every spectrum \(E\), the canonical map
\[
    E^{\mathrm{disc}}
    \longrightarrow
    \widehat{\widehat{E^{\mathrm{disc}}}}
\]
is an equivalence of condensed spectra.
\end{corollary}

\begin{proof}
The homotopy groups of \(E^{\mathrm{disc}}\) are the discrete
condensed abelian groups \((\pi_nE)^{\mathrm{disc}}\), which are
locally compact. The claim therefore follows from
\Cref{thm:condensed-brown-comenetz-duality}.
\end{proof}

\begin{remark}
For example, condensed duality gives
\[
    \widehat{H(\mathbb{Z}^{\mathrm{disc}})}
    \simeq H\mathbb{T},
    \qquad
    \widehat{H\mathbb{T}}
    \simeq H(\mathbb{Z}^{\mathrm{disc}}).
\]
By comparison, the classical Brown--Comenetz double dual of
\(H\mathbb{Z}\) is \(H(\prod_p\mathbb{Z}_p)\), with biduality
map induced by profinite completion. Retaining the topology
on the character group restores the original spectrum.
\end{remark}

\newpage

\bibliographystyle{alphaurl} 
\bibliography{meta/references} 

\end{document}